\documentclass[12pt,oneside,english]{amsart}
\usepackage[T1]{fontenc}
\usepackage[latin9]{inputenc}
\usepackage{fullpage}
\usepackage{amsthm}
\usepackage{amstext}
\usepackage{amssymb}

\usepackage{hyperref}
\hypersetup{colorlinks,%
citecolor=green,%
linkcolor=cyan,%
pdftex}
\usepackage{color}

\newtheorem{theorem}{Theorem}[section]
\newtheorem{proposition}[theorem]{Proposition}
\newtheorem{lemma}[theorem]{Lemma}

\newtheorem{conjecture}[theorem]{Conjecture} 

\theoremstyle{remark}
\newtheorem{definition}[theorem]{Definition}

\newtheorem*{remark}{Remark}

\makeatletter
\numberwithin{equation}{section}
\numberwithin{figure}{section}

\newcommand{\Area}{\mathop{\rm Area}}
\newcommand{\diam}{\mathop{\rm diam}}
\newcommand{\A}{{\mathcal A}}
\newcommand{\B}{{\mathcal B}}
\newcommand{\C}{{\mathcal C}}
\renewcommand{\P}{{\mathbb P}}
\newcommand{\R}{{\mathbb R}}
\newcommand{\T}{{\mathcal T}}

\renewcommand{\mod}[1]{{\ifmmode\text{\rm\ (mod~$#1$)}\else\discretionary{}{}{\hbox{ }}\rm(mod~$#1$)\fi}}
\newcommand{\leg}[2]{\genfrac{(}{)}{}{}{#1}{#2}}

\makeatother

\usepackage{babel}

\title{Primitive points in rational polygons}

\author{Imre B\'ar\'any}
\address{R\'enyi Institute of Mathematics, Hungarian Academy of Sciences, \newline
H-1364 Budapest, Pf. 127, Hungary and \newline
Department of Mathematics, University College London, \newline
Gower Street, London, WC1E 6BT, England} 
\email{barany@renyi.hu}

\author{Greg Martin}
\address{Department of Mathematics, University of British Columbia \newline
Room 121, 1984 Mathematics Road, Vancouver, BC, Canada  V6T 1Z2} 
\email{gerg@math.ubc.ca}

\author{Eric Naslund}
\address{Department of Mathematics, Princeton University \newline
  Fine Hall, Washington Road, Princeton, NJ 08544} 
\email{naslund@princeton.edu}

\author{Sinai Robins}
\address{Sinai Robins, Instituto de Mathematica e Estatistica, Universidade de S\~ao Paulo,  \newline
 05508-090 S\~ao Paulo, Brazil}
\email{srobins@ime.usp.br}


\begin{document}

\begin{abstract} Let $\A$ be a star-shaped polygon in the plane, with rational vertices, containing the origin. The number of primitive lattice points in the dilate $t\A$ is asymptotically $\frac6{\pi^2} \Area(t\A)$ as $t\to \infty$. We show that the error term is both $\Omega_\pm\big( t\sqrt{\log\log t} \big)$ and $O(t(\log t)^{2/3}(\log\log t)^{4/3})$. Both bounds extend (to the above class of polygons) known results for the isosceles right triangle, which appear in the literature as bounds for the error term in the summatory function for Euler's $\phi(n)$.
\end{abstract}

\maketitle

\section{Introduction}

One of the fundamental problems in discrete geometry is to estimate the number of lattice points contained in a polygon. In this paper we concern ourselves with the set of {\em primitive lattice points}
\begin{equation}
\P = \{ (m,n) \in \mathbb Z^2 \colon \gcd(m,n) = 1 \},
\end{equation}
also known as lattice points visible from the origin. It is a classical result the number of primitive lattice points in a ``reasonable'' region in $\R^2$ is approximately $\frac6{\pi^2}$ times the area of the region. We shall be interested in the family $\{t\A\}$ of dilates of a fixed polygon $\A$, for which we define the error term
\begin{equation} \label{eq: E definition}
E_{\A}(t) = \#(t\A\cap\P) - \frac6{\pi^2} \Area(t\A) = \#(t\A\cap\P) - \frac6{\pi^2} t^2 \Area(\A).
\end{equation}
The fact that $\#(t\A\cap\P) \sim \frac6{\pi^2} t^2 \Area(\A)$, or equivalently that $E_{\A}(t) = o(t^2)$, was likely used as far back as Minkowski (see \cite[page 998]{Rogers} or \cite[Theorem 459]{HW}). Stronger upper bounds for $E_{\A}(t)$ are known; we state the following result, which will be justified in the next section, as a benchmark for comparison. 

\begin{proposition} \label{prop: t log t upper bound corollary}
If $\A\subset\R^2$ is a polygon, then
\begin{equation*}
E_{\A}(t) \ll t\log t
\end{equation*}
for $t\ge2$, where the implicit constant may depend on~$\A$.
\end{proposition}

\noindent The purpose of this paper is to improve this upper bound for any {\em rational polygons} (a polygon all of whose vertices have both coordinates rational), and show that it cannot be improved too much more by providing a strong $\Omega_{\pm}$ result for the error term.

It is instructive to consider the specific example $\A=\Delta$, where $\Delta$ is the isosceles right triangle with vertices $(0,0)$, $(1,0)$, and $(1,1)$ and thus area $\frac12$. Then $\#(t\Delta\cap\P)$ is the number of primitive points in the dilate $t\Delta = \{ (x,y)\in \R^2\colon 0\le y\le x\le t\}$, that is,
\begin{equation}\label{eq: totient sum introduced}
\#(t\Delta\cap\P) = \sum_{0\le m\le t} \sum_{\substack{0\le n\le m \\ \gcd(m,n)=1}} 1 = 1 + \sum_{1\le m\le t} \phi(m)
\end{equation}
(where the extra $1$ comes from the fact that $\gcd(0,1)=1$). It is well known that this summatory function of the Euler $\phi$-function is asymptotic to $\frac3{\pi^2}t^2$, so we define
\begin{equation} \label{eq: E delta def}
E_\Delta(t) = \sum_{1\le m\le t} \phi(m) - \frac3{\pi^2} t^2 + 1 = \#(t\Delta\cap\P) - \frac3{\pi^2} t^2.
\end{equation}
Proposition~\ref{prop: t log t upper bound corollary} implies the estimate $E_\Delta(t) \ll t\log t$, which (when phrased in terms of the summatory function of the Euler $\phi$-function) is a classical result of Mertens \cite{Mertens}. The best unconditional upper bound for this error term $E_\Delta(t)$ is due to Walfisz \cite[page 144, eq.\ (3)]{Walfisz}. Using methods related to exponential sums, he showed that
\begin{equation}  \label{eqn: Walfisz}
E_\Delta(t) \ll t(\log t)^{2/3}(\log\log t)^{4/3}.
\end{equation}
As a consequence of our work, we can extend this bound from the isosceles right triangle $\Delta$ to all rational polygons.

\begin{theorem}\label{thm: polygon upper bound}
Let $\A\subset\R^2$ be a rational polygon. Then 
\[
E_{\A}(t)\ll t(\log t)^{2/3}(\log\log t)^{4/3},
\]
where the implicit constant may depend on~$\A$.
\end{theorem}

Prior to Walfisz's result, Chowla and Pillai \cite{CP} showed that these upper bounds cannot be improved too much by establishing the lower bound
\[
\limsup_{t\to\infty} \frac{|E_\Delta(t)|}{t\log\log\log t} > 0,
\]
that is, $E_\Delta(t) = \Omega(t\log\log\log t)$. Later, Erd\H os and Shapiro \cite{ES} obtained a slightly weaker quantitative lower bound but showed that $E_\Delta(t)$ oscillates in sign infinitely often. Both results were improved by Montgomery \cite{Montgomery87}:

\begin{proposition} \label{prop: Montgomery}
$E_\Delta(t)=\Omega_{\pm}(t\sqrt{\log\log t})$; in other words,
\[
\limsup_{t\to\infty} \frac{E_\Delta(t)}{t\sqrt{\log\log t}} > 0 \quad\text{and}\quad \liminf_{t\to\infty} \frac{E_\Delta(t)}{t\sqrt{\log\log t}} < 0.
\]
\end{proposition}

We use the term {\em origin-star-shaped} (star-shaped with respect to the origin) to refer to any domain $\B$ for which $\lambda \B \subset \B$ for all $\lambda \in [0,1]$; note that in particular, any origin-star-shaped domain contains the origin. The main result of this paper generalizes this lower bound of Montgomery to this class of polygons:

\begin{theorem}\label{thm: polygon lower bound}
Let $\A\subset\R^2$ be a rational origin-star-shaped polygon. Then 
\[
E_{\A}(t)=\Omega_{\pm}(t\sqrt{\log\log t}),
\]
where the implicit constant may depend on~$\A$.
\end{theorem}

In the process of applying an adaptation of Montgomery's argument to the general error term $E_{\A}(t)$ for the primitive point counting function for the lattice polygon $\A$, we found ourselves needing to establish the following ``error term independence'' result concerning the error term $E_\Delta$ for the summatory function of $\phi(n)$, which was defined in equation~\eqref{eq: E delta def}; this result may be of independent interest.

\begin{theorem} \label{thm: error term independence}
For any positive rational numbers $c_1,\dots,c_k$ and $f_1,\dots,f_k$,
\[
c_1 E_\Delta(f_1 x) + \cdots + c_k E_\Delta(f_k x) = \Omega_\pm\big( x\sqrt{\log\log x} \big),
\]
where the implied constant may depend upon the $c_j$ and $f_j$.
\end{theorem}
\noindent In other words, there can be no ``magic cancellation'' among the terms $E_\Delta(f_j x)$ that makes the oscillation of the sum significantly smaller than that of an individual term. Furthermore we conjecture that oscillations of the same size exist even when the $c_i$ are allowed to be negative (where we require the $f_i$ to be distinct to avoid trivial cancellations). 
\begin{conjecture}
Theorem \ref{thm: error term independence} holds for any rational numbers $c_i$ and any distinct positive rational numbers $f_i$.
\end{conjecture}

\noindent This would imply a stronger version of Theorem \ref{thm: polygon lower bound} that holds for any rational polygon $\mathcal{A}$, not necessarily star-shaped or containing the origin:

\begin{conjecture}
Let $\A\subset\R^2$ be a rational polygon. Then 
\[
E_{\A}(t)=\Omega_{\pm}(t\sqrt{\log\log t}),
\]
where the implicit constant may depend on~$\A$.
\end{conjecture}

The rest of the paper is divided into two sections. In the next section, we show (Theorem~\ref{thm: error term decomposition}) that the error term $E_{\A}(t)$ may be rewritten as a linear combination of dilates of the totient error function $E_\Delta(t)$, thereby establishing Theorem~\ref{thm: polygon upper bound} and reducing Theorem~\ref{thm: polygon lower bound} to Theorem~\ref{thm: error term independence}. We then establish this latter theorem in the final section.

\section{Decomposing the error term}\label{sec: Error term}

We begin by giving a proof of Proposition~\ref{prop: t log t upper bound corollary}, not only for the sake of completeness, but also because the structure of the argument foreshadows how we will approach the main result of this section, namely Theorem~\ref{thm: error term decomposition} below. We use the term \emph{pointed triangle} to mean a triangle which has the origin as a vertex.


\begin{proof}[Proof of Proposition~\ref{prop: t log t upper bound corollary}]
We begin by quoting a reasonably precise estimate~\cite[special case of Theorem 2.1]{KrPo} for the number of primitive points inside a  domain $\B$: if $\B$ is convex and contains the origin, then
\begin{equation} \label{must contain 0}
\#(\B\cap\P) - \frac6{\pi^2} \Area(\B) \ll \max\{1, \omega \log \omega\},
\end{equation}
where $\omega$ is the diameter of $\B$. In particular, if $\C$ is a convex polygon containing the origin and $\B=t\C$ is a dilate, then the diameter of $\B$ is a constant multiple of $t$ and so 
\[
E_\C(t) \ll \max\{1,t\log t\},
\]
with the implicit constant depending on~$\C$.

However, any polygon $\A$ can be partitioned into sums and differences of finitely many pointed triangles, simply by triangulating $\A$ and noting that any triangle can be written as the signed sum of three pointed triangles. Let $\{\C_j\}_{j=1}^{n}$ denote these triangles, and $\epsilon_j\in\{-1,1\}$ the sign of $\C_j$, that is whether it is added or substracted, so that $\Area(A)=\sum_{j=1}^n \epsilon_j \Area(\C_j)$. Even though these triangles have sides in common, the number of double-counted lattice points on the $t$-dilation of each side is at most $O(t)$, and therefore the above bound (valid since pointed triangles are certainly convex polygons containing the origin) implies
\begin{align*}
E_{\A}(t) &= \#(t\A\cap\P) - \frac6{\pi^2}\Area(t\A) = \bigg( \sum_{j=1}^n \epsilon_j \#(t\C_j\cap\P) + O(t) \bigg) -\frac6{\pi^2}\Area(t\A) \\
&= \sum_{j=1}^n \epsilon_j\bigg( \frac6{\pi^2}\Area(t\C_j) + E_{\C_j}(t) \bigg) + O(t) - \frac6{\pi^2}\Area(t\A) \\
&= \sum_{j=1}^n \epsilon_j E_{\C_j}(t) + O(t) \ll t\log t
\end{align*}
for $t\ge2$, as desired.
\end{proof}

\begin{remark}
In~\cite{KrPo}, the bound~\eqref{must contain 0} is stated without the hypothesis that $\B$ contain the origin; however, this hypothesis is actually necessary. One can easily construct, using the Chinese remainder theorem, a square $\B$ of diameter $\omega$ containing no primitive lattice points whatsoever. The area of such a square is a constant times $\omega^2$, and therefore
$$
\big| \#(\B\cap\P) - \frac6{\pi^2}\Area(\B) \big| \gg \omega^2,
$$ which is incompatible with the claimed bound $\omega\log\omega$.

The error in the proof of~\cite[Theorem 2.1]{KrPo} comes when applying Lemma 2.3, which requires $\omega\ge1$, to a term of the form $\sum_{(\Delta/k)_1} f(kx)$; when $k>\omega$, the diameter of $\Delta/k$ is less than $1$, and so Lemma 2.3 cannot be applied. However, if $\Delta$ is a set containing the origin, then these sets $\Delta/k$ are sets with diameter less than $1$ that contain the origin, hence contain no primitive lattice points (indeed, no lattice points at all other than the origin). Therefore the sums over $k$ in the proof of~\cite[Theorem 2.1]{KrPo} can be truncated at $k\le \omega$, and the rest of the argument goes through thereafter. \end{remark}


The remainder of this section is dedicated to proving Theorem~\ref{thm: error term decomposition}, which asserts that the error term $E_{\A}(t)$ for any rational polygon $\A$ may be rewritten in terms of the totient error function $E_\Delta(t)$. We begin with a detailed investigation of this latter function.

\begin{definition}
\label{def: fractional notation}
Throughout this section, we employ the notation $\lfloor x\rfloor$ for the greatest integer not exceeding $x$ and $\{x\} = x-\lfloor x\rfloor$ for the fractional part of $x$. We also employ the sawtooth function defined as 
\[
B(x)= \{x\} -\frac{1}{2}
\]
for all real numbers $x$; this function is equal to the first periodic Bernoulli polynomial $\bar{B}_1(x)$ except at integer arguments. 
\end{definition}

Next we recall some well-known and useful estimates for sums involving the M\"obius $\mu$-function, which satisfies the key identity
\begin{equation}\label{eq: key M\"obius identity}
\sum_{d|n}\mu(d)=\begin{cases}
1, & \text{if }n=1,\\
0, & \text{otherwise.}
\end{cases}
\end{equation}

For any real numbers $2\le x \le y$:
\begin{align}
\sum_{d\le x} \frac{\mu(d)}{d^2} = \frac6{\pi^2} + O\bigg( \frac1x \bigg);\\
\sum_{d\le y} \mu(d) \bigg\lfloor \frac xd \bigg\rfloor = 1;\\
\bigg| \sum_{d\le x} \frac{\mu(d)}{d} \bigg| \ll 1.
\end{align}

The proof of (2.3) and (2.5) is in \cite{Jameson} and (2.4) can be found in \cite{Apostol}, for instance. We are now ready to give a more precise formula for the totient error term $E_\Delta(t)$.

\begin{proposition} \label{prop: formula for error term}
For any real numbers $T\ge t\ge2$,
\begin{equation*}
E_{\Delta}(t)=-t\sum_{d\le T}\frac{\mu(d)}{d}B\bigg( \frac{t}{d}\bigg) +O(T).
\end{equation*}
\end{proposition}

\begin{proof}
Starting from equation~\eqref{eq: totient sum introduced} and the key identity~\eqref{eq: key M\"obius identity}, we have 

\begin{align*}
\#(t\Delta\cap\P) &= 1 + \sum_{1\le m\leq t}\sum_{\substack{
1\le n\leq m\\
\gcd(m,n)=1
}}1 = \sum_{m\leq t}\sum_{\substack{
n\leq m\\
\gcd(m,n)=1
}}\sum_{d\mid\gcd(m,n)}\mu(d) \\
&= \sum_{d\leq t}\mu(d)\sum_{\substack{
m\leq t\\
d\mid m
}}\sum_{\substack{
n\leq m\\
d\mid n
}}1 = \sum_{d\leq t}\mu(d)\sum_{k\le t/d} \sum_{\ell\le k}1 \\
&= \sum_{d\le t}\mu(d)\sum_{k\le t/d} k = \sum_{d\le t}\mu(d)\frac{1}{2}\bigg(\bigg\lfloor\frac{t}{d}\bigg\rfloor^{2}+\bigg\lfloor\frac{t}{d}\bigg\rfloor\bigg) \\
&= \frac{1}{2}\sum_{d\le t}\mu(d)\bigg\lfloor\frac{t}{d}\bigg\rfloor^{2}+\frac{1}{2}
\end{align*}
by (2.4). The last sum does not change if we take it for all $d\le T$ instead of $d\le t$. Expanding $\lfloor \frac td \rfloor^{2}=(\frac td - \{\frac td\})^{2}$ we obtain 
\begin{align*}
\#(t\Delta\cap\P) &= \frac{t^{2}}{2}\sum_{d\le T}\frac{\mu(d)}{d^{2}}-t\sum_{d\le T}\frac{\mu(d)}{d}\bigg\{ \frac{t}{d}\bigg\} + \sum_{d\le T} \mu(d) \bigg\{ \frac{t}{d}\bigg\}^2 +\frac12 \\
&= \frac{t^{2}}{2} \bigg( \frac6{\pi^2} + O\bigg(\frac1T \bigg) \bigg) -t\sum_{d\le T}\frac{\mu(d)}{d}\bigg\{ \frac{t}{d}\bigg\} +O(T)
\end{align*}
by (2.3) and a trivial bound. Therefore
\begin{align*}
E_{\Delta}(t) &= \#(t\Delta\cap\P) - \frac{6}{\pi^{2}}t^{2}\Area(\Delta) 
= \frac{3}{\pi^{2}}t^{2}-t\sum_{d\le T}\frac{\mu(d)}{d}\bigg\{ \frac{t}{d}\bigg\} +O(T)  -\frac{3}{\pi^{2}}t^{2} \\
&= -t\sum_{d\le T}\frac{\mu(d)}{d}\bigg( B\bigg( \frac{t}{d}\bigg) + \frac12 \bigg) +O(T) 
= -t\sum_{d\le T}\frac{\mu(d)}{d} B\bigg( \frac{t}{d}\bigg) +O(T)
\end{align*}
by (2.5).
\end{proof}

Counting lattice points in the simplest way (each with weight $1$) is problematic in our present context.   First, we will be decomposing polygons into unions of triangles that share sides, and so double-counting lattice points on these shared boundaries would become an issue; second, any error term as large as the perimeter for counting lattice points in a triangle would result in an unacceptably large error term in the inclusion-exclusion method we use to detect primitive lattice points. Consquently, we employ a more convenient weighting in our lattice-point counting, namely the solid angle sum of a polygon.

%
%

\begin{definition}
For a point $p\in\mathbb{R}^2$, let $B(p;r)=\{ q:\ d(p-q)<r\}$ denote the ball of radius $r$, centered at $p$. Let $\lambda$ denote Lebesgue measure on $\mathbb{R}^{2}$, and define for any polygon $\A$
\[
\omega_{\A}(p)=\lim_{r\rightarrow0}\frac{\lambda(B(p;r)\cap\A)}{\lambda(B(p;r))}. 
\]
It follows directly from the definition that
\[
\omega_{\A}(p)=
\begin{cases}
0, & \text{if }p\notin\A, \\
1, & \text{if }p\text{ is in the interior of }\A, \\
1/2, & \text{if }p\text{ lies in the interior of an edge of } \A, \\
\theta_{p}/2\pi, & \text{if }p\text{ is a vertex of }\A,
\end{cases} 
\]
where $\theta_p$ is the angle at the vertex $p$.
Using this function, we define the solid angle sum of the $t$-dilate of the rational polygon $\A$ to be
\[
A_{\A}(t)=\sum_{p\in\mathbb{Z}^2} \omega_{t\A}(p).
\]
We also define the corresponding sum over primitive lattice points only:
\[
\P_{\A}(t)=\sum_{p\in\P}\omega_{t\A}(p).
\]
The benefit of the solid angle weighting is that both of these functions are additive, in the sense that $A_{\A\cup\B}(t) = A_{\A}(t) + A_{\B}(t)$ and $\P_{\A\cup\B}(t) = \P_{\A}(t) + \P_{\B}(t)$ for any polygons $\A$ and $\B$ with disjoint interiors.
\end{definition}

\begin{proposition}  \label{prop: solid angle sum formula}
Let $\A$ be a rational polygon. There exist a real number $C$, a positive integer $k$, and rational numbers $c_1,\dots,c_k,f_1,\dots,f_k$ where $f_i>0$ (all depending on $\A$) such that
\begin{equation} \label{eq: solid angle sum formula}
A_{\A}(t)=\Area(\A)t^2+Ct-t\sum_{j=1}^k c_j B(f_j t)+O(1),
\end{equation}
where the implicit constant may depend upon $\A$. Furthermore, when $\mathcal{A}$ is an origin-star-shaped polygon the $c_j$ may be taken to all be positive.
\end{proposition}

\begin{proof}
Le Quang and Robins \cite{LR} gave a precise formula for the solid angle sum of any rational pointed triangle 
$\mathcal{T}$:
\begin{equation}  \label{LQR}
A_{\T}(t)=\Area(\T)t^2+C(\T)t-t\sum_{j=1}^{k(\T)} c_j(\T) B(f_j(\T) t)+O(1)
\end{equation}
where $C(\T)$, $k(\T)$, $c_j(\T)$, $f_j(\T)$ are constants depending on $\mathcal{T}$ with $k(\T)$ a positive
 integer, and the $c_j(\T)$ and $f_j(\T)$ positive rational numbers. But the rational polygon $\A$ can be partitioned as a signed sum of a finite number of rational pointed triangles (as in the proof of Proposition~\ref{prop: t log t upper bound corollary}). Simply summing the formula~\eqref{LQR} over these finitely many triangles yields the desired formula~\eqref{eq: solid angle sum formula}, where $C$ is the sum of the $C(\T)$, and $\{f_j\}$ is the union of the $\{f_j(\T)\}$, and so on.
\end{proof}

The following theorem results from a careful examination of $\P_{\A}(t)$, which can be related to the solid angle sum $A_{\A}(t)$ using the M\"obius function as in the proof of Proposition~\ref{prop: formula for error term}.

\begin{theorem} \label{thm: error term decomposition}
Let $\A$ be a rational polygon. There exist a positive integer $k$ and rational numbers $r_1,\dots,r_k,f_1,\dots, f_k$ where $f_i>0$ such that
\begin{equation} \label{eq: error term decomposition}
E_{\A}(t)=\sum_{j=1}^{k}r_j E_\Delta(f_j t)+O(t),
\end{equation}
where the implicit constant may depend upon $\A$. Furthermore, if $\A$ is an origin-star-shaped polygon, then the $r_j$ may all be taken to be positive.
\end{theorem}

\begin{remark}
Theorem~\ref{thm: polygon upper bound} follows immediately from Theorem~\ref{thm: error term decomposition} in light of the known upper bound~\eqref{eqn: Walfisz}. In addition, Theorem \ref{thm: polygon lower bound} follows immediately from the combination of Theorem~\ref{thm: error term decomposition} and Theorem \ref{thm: error term independence}; the latter theorem is proved in the next section.
\end{remark}

\begin{proof}
We commence by excluding the (non-primitive) origin and using the M\"obius identity (\ref{eq: key M\"obius identity}) to write
\[
\P_{\A}(t) = \sum_{(m,n)\in\mathbb{Z}^{2}\setminus(0,0)}\omega_{t\A}(m,n) \sum_{d\mid (m,n)}\mu(d).
\]
Interchanging the order of summation (valid since in reality there are only finitely many nonzero terms) and rescaling (which preserves solid angles),
\begin{align*}
\P_{\A}(t) &= \sum_{d=1}^\infty \mu(d)\sum_{\substack{(m,n)\in\mathbb{Z}^{2}\setminus(0,0) \\ d\mid m,\, d\mid n}}\omega_{t\A}(m,n) 
= \sum_{d=1}^\infty \mu(d)\sum_{(x,y)\in\mathbb{Z}^{2}\setminus(0,0)}\omega_{\frac{t}{d}\A}(x,y) \\
&= \sum_{d=1}^\infty \mu(d)\bigg(A_{\A} \bigg(\frac{t}{d}\bigg)-\omega_{\frac td\A}(0,0)\bigg)
\end{align*}
by the definition of $A_{\A}$, We may truncate the outer sum at any real number $T\ge t\diam(\A)$ is allowed because for larger values of $d$, the diameter of $\frac td\A$ is less than $1$, and hence $\frac td\A$ (which contains the origin) cannot contain any other lattice points. Proposition~\ref{prop: solid angle sum formula} now implies
\begin{align*}
\P_{\A}(t) &= \sum_{d\le T} \mu(d)\bigg(\Area(\A)\frac{t^{2}}{d^{2}}+C\frac{t}{d}-\frac{t}{d}\sum_{j=1}^{k}c_j B\bigg(f_j\frac{t}{d}\bigg)+O(1)\bigg) \\
&= \Area(\A)t^{2}\sum_{d\le T}\frac{\mu(d)}{d^{2}}+Ct\sum_{d\le T}\frac{\mu(d)}{d} \\
&\qquad{} -t\sum_{j=1}^{k}c_j\sum_{d\le T}\frac{\mu(d)}{d}B\bigg(f_j\frac{t}{d}\bigg)+O\bigg(\sum_{d\le T} |\mu(d)|\bigg) \\
&= \Area(\A)t^{2} \bigg( \frac6{\pi^2} + O\bigg( \frac1T \bigg) \bigg) + O(|C|t\cdot 1) -t\sum_{j=1}^{k}c_j\sum_{d\le T} \frac{\mu(d)}{d}B\bigg(f_j\frac{t}{d}\bigg)+O(T),
\end{align*}
Using Proposition~\ref{prop: formula for error term}, we may rewrite the above in terms of $E_\Delta(t)$ and obtain 
\begin{equation}  \label{eqn: formula for P_A(t)}
\P_{\A}(t)=\Area(\A)\frac{6}{\pi^{2}}t^{2}+\sum_{i=1}^{k}\frac{c_{i}}{f_{i}}E_{\Delta}(f_{i}t)+O(T)
\end{equation}
for any $T\ge t\max\{1,\diam(\A),f_1,\dots,f_k\}$.

The number of integer points on the boundary of $t\A$ is $O(t)$ since $t\A$ has finitely many sides, each of which has length $O(t)$. Consequently, 
\begin{equation*} 
\P_{\A}(t)-\#(t\A\cap\P) \ll t,
\end{equation*}
and therefore equation~\eqref{eqn: formula for P_A(t)} implies
\begin{align*}
E_{\A}(t)&=\#(t\A\cap\P)-\Area(\A)\frac{6}{\pi^{2}}t^2 \\
&=\P_{\A}(t)+O(t)-\Area(\A)\frac{6}{\pi^{2}}t^2 =\sum_{j=1}^{k}\frac{c_j}{f_j}E_{\Delta}(f_jt)+O(T).
\end{align*}
Upon setting $T= t\max\{1,\diam(\A),f_1,\dots,f_k\}$ and $r_j=c_j/f_j$ for each $1\le j\le k$, the theorem follows.
\end{proof}

\section{Linear combinations of $E_\Delta$}

In this section we prove Theorem \ref{thm: error term independence}, showing that positive rational linear combinations of scaled copies of the totient error function $E_\Delta$ have oscillations as large as those known for $E_\Delta$ itself. We begin by recalling some of the components of Montgomery's argument~\cite{Montgomery87} establishing the oscillations of $E_\Delta$, after which we describe the strategy that led to our modifications.

\begin{definition}
\label{R def}
Define
\[
R_0(x)= \sum_{n\le x} \frac{\phi(n)}n - \frac{6}{\pi^{2}}x.
\]
\end{definition}

Montgomery~\cite[Theorem 1]{Montgomery87} showed that the totient error function is closely connected to the above weighted error:

\begin{lemma}[Montgomery] \label{R0 to EDelta}
$E_\Delta(x) = xR_0(x) + O\big(x \exp(-c\sqrt{\log x})\big)$.
\end{lemma}

\begin{definition}
\label{Montgomery quantities def}
Define
\[
K(q,\alpha) = {-} \sum_{d\mid q}\frac{\mu(d)}{d} B\bigg( \frac{\alpha}{d} \bigg),
\]
where the sawtooth function $B(x)$ was defined in Definition~\ref{def: fractional notation}, and
\[
C(q,\alpha) = K(q,\alpha) \frac{6}{\pi^{2}} \prod_{p\mid q} \bigg(1-\frac1{p^2} \bigg)^{-1}.
\]
\end{definition}

\begin{lemma} \label{K scale lemma}
If $b$ is relatively prime to $q$, then
\begin{equation*}
K(qb,\alpha b) = \sum_{d_1d_2=b}\frac{\mu(d_1)}{d_1} K(q,\alpha d_2).
\end{equation*}
\end{lemma}

\begin{proof}
Since every divisor of $qb$ can be written uniquely as a divisor of $b$ times a divisor of $q$,
\begin{equation*}
K(qb,\alpha b) = {-} \sum_{d\mid qb}\frac{\mu(d)}{d} B\bigg( \frac{\alpha b}{d} \bigg) = {-} \sum_{a\mid b}\frac{\mu(a)}{a} \sum_{c\mid q}\frac{\mu(c)}{c} B\bigg( \frac{\alpha b}{ac} \bigg) = \sum_{a\mid b}\frac{\mu(a)}{a} K\bigg(q,\frac{\alpha b}a\bigg),
\end{equation*}
which is equivalent to the statement of the lemma.
\end{proof}

The above quantities appear~\cite[Lemma 4]{Montgomery87} in a key part of Montgomery's argument, which displays a bias in the values of $R_0$ sampled on an arithmetic progression:

\begin{lemma}[Montgomery] \label{Montgomery constant lemma}
There exists a positive real number $c$ such that if $\alpha$ is a non-integral real number with $0<\alpha<q$, then 
\[
\sum_{n=1}^{N}R_0(nq+\alpha)=C(q,\alpha)N+O\big(N\exp(-c\sqrt{\log N})\big)
\]
uniformly for $q\leq e^{c\sqrt{\log N}}$.
\end{lemma}

The proof of Lemma~\ref{Montgomery constant lemma} is based on the following ancient identity due to Raabe~\cite{Raabe} for a sum of $B( x )$ over an arithmetic progression of points: for all real numbers $x$,
\begin{equation*} \label{Bernoulli identity}
 \sum_{j=1}^J B\bigg( x + \frac jJ \bigg)=B(Jx).
\end{equation*}
valid for all real $x$; the relevance to the problem at hand is due to the fact~\cite[Lemma 1]{Montgomery87} that
\[
R_0(x) = {-} \sum_{d\le x} \frac{\mu(d)}d B\bigg( \frac xd \bigg) + O(1).
\]

To exploit Lemma~\ref{Montgomery constant lemma}, Montgomery chose $\alpha=q/4$, so that for all divisors $d$ of $q$ the quantity $B(\alpha/d)$ equals $\pm\frac14$, with the sign depending only upon the residue class of $d$ modulo~$4$. On the other hand, he chose $q$ to be the product of many primes congruent to $3$ modulo $4$; for any divisor $d$ of this $q$, the residue class of $d\mod 4$ depends only on the number of prime factors of~$d$. With these choices, the sign of $B(\alpha/d)$ correlates exactly with the sign of $\mu(d)$, making the quantity $K(q,\alpha)$ large in absolute value. Choosing $\alpha=3q/4$ instead again makes $K(q,\alpha)$ large but with the opposite sign.

One of our key observations is that we may work modulo a suitably chosen prime $P$ rather than working modulo~$4$. Instead of choosing $q$ to be the product of many primes congruent to $3\mod 4$, we instead choose $q$ to be the product of many primes that are quadratic nonresidues modulo~$P$. The sign of $B(\alpha/d)$ will not be perfectly correlated with $\mu(d)$, but there will be enough of a systematic bias in the signs of $B(\alpha/d)$ (due to the imperfect distribution of quadratic residues and nonresidues modulo~$P$) that we can still force $K(q,\alpha)$ to be large in absolute value. If we choose the prime $P$ carefully, we can even force all of the different $K(qb_{i},\alpha b_{i})$ to be large in absolute value and have the same sign.

We introduce the following function, whose oscillations we will want to establish. For the rest of this section, we use boldface variables such as $\bf a$ to indicate the dependence of various quantities on $k$-tuples $(a_1,\dots,a_k)$ of variables.

\begin{definition}
\label{G def}
Given real numbers $a_1,\dots,a_k$ and $b_1,\dots,b_k$, define
\[
G_{\bf a,b}(x)=a_{1}R_0(b_{1}x)+\cdots+a_{k}R_0(b_{k}x).
\]
\end{definition}

\noindent The starting point of our modification of Montgomery's method is the following easy consequence of Lemma~\ref{Montgomery constant lemma}.

\begin{lemma}
\label{lem: montgomery's lemma for G in progressions}
Let $a_1,\dots,a_k$ and $b_1,\dots,b_k>0$ be real numbers. There exists a positive real number $c$ such that: if $\alpha \in (0,q)$ is a real number such that none of the $\alpha b_j$ is an integer, then 
\[
\sum_{n=1}^{N}G_{\bf a,b}(nq+\alpha)=N\sum_{j=1}^{k}a_jC(qb_j, \alpha b_j)+O_{\bf a,b}(N\exp(-c\sqrt{\log N}))
\]
uniformly for $\max\{b_1,\dots,b_k\}<q\le e^{c\sqrt{\log N}}$.
\end{lemma}


We now state our oscillation result for $G_{\bf a,b}$, after which we show how Theorem~\ref{thm: error term independence} is implied by it. Thereafter our only remaining goal will be to establish this proposition:

\begin{proposition}
\label{prop: Totient key proposition}
Let $a_1,\dots,a_k$ and $b_1,\dots,b_k$ be fixed positive integers. There exists a constant $\kappa_{\bf a,b}>0$, and sequences $Q_{N,\bf b}\le e^{\sqrt{\log N}}$ and $0<\alpha_N^+,\alpha_N^-<Q_{N,\bf b}$ defined for positive integers $N$, for which $Q_{N,\bf b}$ tends to infinity with $N$ and
\[
\sum_{n=1}^{N}G_{\bf a,b}(nQ_{N,\bf b}+\alpha_N^+)=\kappa_{\bf a,b} N\sqrt{\log\log N}+O_{\bf a,b}(N)
\]
and 
\[
\sum_{n=1}^{N}G_{\bf a,b}(nQ_{N,\bf b}+\alpha_N^-)=-\kappa_{\bf a,b} N\sqrt{\log\log N}+O_{\bf a,b}(N).
\]
\end{proposition}

\begin{proof}[Proof of Theorem~\ref{thm: error term independence} assuming Proposition~\ref{prop: Totient key proposition}]
By Lemma~\ref{R0 to EDelta},
\[
\sum_{j=1}^k r_j E_\Delta(s_j x) = x \sum_{j=1}^k r_j s_j R_0(s_j x) + O_{\bf r,s}\big(x \exp(-c\sqrt{\log x})\big).
\]
Let $D_{\bf s}$ be the least common denominator of the rational numbers $s_1,\dots,s_k$, and set $a_j=D_{\bf s}r_js_j$ and $b_j=D_{\bf s}s_j$. Replacing $x$ by $D_{\bf s}x$, we obtain
\begin{equation}  \label{converting E to G}
\sum_{j=1}^k r_j E_\Delta(s_j D_{\bf s}x) = x G_{\bf a,b}(x) + O_{\bf r,s}\big(x \exp(-c\sqrt{\log x})\big).
\end{equation}
Let $0<\varepsilon<\kappa_{\bf a,b}$. If there existed an $x_0$ such that $G_{\bf a,b}(x) < (\kappa_{\bf a,b}-\varepsilon)\sqrt{\log\log(D_{\bf s}x)}$ for all $x>x_0$, then we would have
\begin{align*}
\sum_{n=1}^N G_{\bf a,b}(nQ_N+\alpha_N^+) &< (\kappa_{\bf a,b}-\varepsilon) N \sqrt{\log\log(D_{\bf s}(NQ_N+\alpha_N^+))} + O(x_0\max_{1\le t\le x_0} G_{\bf a,b}(t)) \\
&= (\kappa_{\bf a,b}-\varepsilon) N \bigg( \sqrt{\log\log N} + O_{\bf s}\bigg( \frac1{\sqrt{\log N}} \bigg) \bigg) + O_{\varepsilon,\bf a,b}(1)
\end{align*}
by the bounds $0<\alpha_N^+<Q_{N,\bf b}$; for large $N$ (and hence for large $Q_{N,\bf b}$) this would contradict Proposition~\ref{prop: Totient key proposition}. Therefore no such $x_0$ can exist, in which case equation~\eqref{converting E to G} implies that there are arbitrarily large values of $x$ for which
\begin{equation*}
\sum_{j=1}^k r_j E_\Delta(s_j D_{\bf s}x) \ge x (\kappa_{\bf a,b}-\varepsilon)\sqrt{\log\log(D_{\bf s}x)} + O_{\bf r,s}\big(x \exp(-c\sqrt{\log x})\big).
\end{equation*}
In other words,
\[
\limsup_{x\to\infty} \frac{r_1E_\Delta(s_1 x) + \cdots + r_kE_\Delta(s_k x)}{x\sqrt{\log\log x}} \ge \frac{\kappa_{\bf a,b}}{D_{\bf s}},
\]
and the analogous argument using the values $G(nQ_N+\alpha_N^-)$ gives
\[
\liminf_{x\to\infty} \frac{r_1E_\Delta(s_1 x) + \cdots + r_kE_\Delta(s_k x)}{x\sqrt{\log\log x}} \le -\frac{\kappa_{\bf a,b}}{D_{\bf s}},
\]
completing the derivation of Theorem~\ref{thm: error term independence} from Proposition~\ref{prop: Totient key proposition}.
\end{proof}

Before addressing Proposition~\ref{prop: Totient key proposition} directly, we record some preliminary facts about the distribution of primes in residue classes and the associated $L$-values.

\begin{lemma}  \label{class number lemma}
Let $P\equiv3\mod4$ be a prime exceeding $3$, and let $\chi_1(\cdot) = \leg\cdot P$ denote the quadratic character modulo~$P$. Then the class number $h(-P)$ of the field $\mathbb Q(\sqrt{-P})$ equals
\[
h(-P) = \frac{\sqrt P}\pi L(1,\chi_1) = -\frac{1}{P}\sum_{a=1}^{P-1}a\chi_1(a).
\]
\end{lemma}

\begin{proof}
These results are classical; see for example~\cite[chapter 6, equations (15) and (19)]{Davenport}.
\end{proof}

\begin{lemma} \label{half product lemma}
Let $P\equiv3\mod4$ be a prime exceeding $3$. If $\chi_0$ denotes the principal character\mod P,  then
\[
\prod_{\substack{p\le y \\ \leg pP=-1}} \bigg( 1 + \frac{\chi_0(p)}p \bigg) = c_P \sqrt{\log y} + O_P(1),
\]
where
\begin{equation} \label{cP def}
c_P = \bigg( \frac{e^\gamma}\pi \frac{P-1}{h(-P)\sqrt P} \prod_{\leg pP=-1} \big(1-p^{-2}\big) \bigg)^{1/2}.
\end{equation}
On the other hand, for any nonprincipal character $\chi\mod P$,
\[
\prod_{\substack{p\le y \\ \leg pP=-1}} \bigg( 1 + \frac{\chi(p)}p \bigg) \ll_P 1.
\]
\end{lemma}

\begin{proof}
Again let $\chi_1(\cdot) = \leg\cdot P$ denote the quadratic character\mod P. For any character $\chi\mod P$ we can write
\begin{equation} \label{three products}
\bigg( \prod_{\substack{p\le y \\ \leg pP=-1}} \bigg( 1 + \frac{\chi(p)}p \bigg) \bigg)^2 = \prod_{p\le y} \bigg( 1 - \frac{\chi(p)}p \bigg)^{-1} \prod_{p\le y} \bigg( 1 - \frac{\chi(p)\chi_1(p)}p \bigg) \prod_{\substack{p\le y \\ \leg pP=-1}} \bigg( 1 - \frac{\chi^2(p)}{p^2} \bigg).
\end{equation}
The last product is absolutely convergent uniformly in $P$; indeed,
\[
\bigg| \prod_{\substack{p>y \\ \leg pP=-1}} \bigg( 1 - \frac{\chi^2(p)}{p^2} \bigg) \bigg| \le \prod_{\substack{p>y \\ \leg pP=-1}} \bigg( 1 + \frac1{p^2} \bigg) < \prod_{n>y} \bigg( 1 + \frac1{n^2} \bigg) < \prod_{n>y} \bigg( 1-\frac1{n^2} \bigg)^{-1} = 1+\frac1{\lceil y\rceil},
\]
and so the last product equals $\prod_{\leg pP=-1} \big(1-\chi^2(p)p^{-2}\big) \big( 1 + O\big( \frac1y \big) \big)$ and in particular is uniformly bounded.
When $\chi$ is nonprincipal, we know~\cite[Theorem 4.11(d)]{MV} that
\[
\prod_{p\le y} \bigg( 1 - \frac{\chi(p)}p \bigg)^{-1} = L(1,\chi) + O_\chi\bigg( \frac1{\log y} \bigg) = L(1,\chi) \bigg( 1 + O_P\bigg( \frac1{\log y} \bigg) \bigg),
\]
since $L(1,\chi)\ne0$~\cite[Theorem 4.9]{MV}. (Better error terms are available but are not relevant for us.) In particular, when neither $\chi$ nor $\chi\chi_1$ is principal, the first two products on the right-hand side of equation~\eqref{three products} are $L(1,\chi) \big( 1 + O_P\big( \frac1{\log y} \big) \big)$ and $L(1,\chi\chi_1)^{-1} \big( 1 + O_P\big( \frac1{\log y} \big) \big)$, respectively; in particular, both are $\ll_P1$. This estimate establishes the lemma when $\chi$ is neither principal nor equal to $\chi_1$.

When $\chi=\chi_1$, the first and third factors are still bounded, while now the second factor actually diverges to $0$, hence in particular is still bounded. Finally, when $\chi=\chi_0$ is principal, the second factor converges to $1/L(1,\chi_1) = \sqrt P/\pi h(-P)$ by Lemma~\ref{class number lemma}, while the first factor on the left-hand side of equation~\eqref{three products} is
\[
\prod_{p\le y} \bigg( 1 - \frac{\chi_0(p)}p \bigg)^{-1} = \prod_{\substack{p\le y \\ p\ne P}} \bigg( 1 - \frac1p \bigg)^{-1} = \frac{P-1}P (e^\gamma \log y )\bigg(1+ O\bigg(\frac1{\log y}\bigg)\bigg)
\]
by Mertens's formula~\cite[Theorem 2.7(e)]{MV}. This asymptotic evaluation of the right-hand side of equation~\eqref{three products} establishes the lemma when $\chi$ is principal, indeed with the stronger error term $O_P\big( \frac1{\sqrt{\log y}} \big)$.
\end{proof}

We can now define the modulus $Q_{N,\bf b}$ appearing in the statement of Proposition~\ref{prop: Totient key proposition}, in terms of a companion prime $P_{\bf b}$.

\begin{definition} \label{P and Q definition}
Let $P_{\bf b}$ be the smallest prime satisfying $P_{\bf b}\equiv-1\mod{8b_1\dots b_k}$. Note that $P_{\bf b}\equiv7\mod8$, and so $-1$ is a quadratic nonresidue\mod{P_{\bf b}} while $2$ is a quadratic residue\mod{P_{\bf b}}. Furthermore, if $p$ is any odd prime dividing one of the $b_j$, then by quadratic reciprocity \cite[Theorem 3.1]{NZM} we have, since $P_{\bf b}\equiv3\mod4$ and $P_{\bf b}\equiv-1\mod p$,
\[
\leg p{P_{\bf b}} = (-1)^{(p-1)/2} \leg {P_{\bf b}}p = (-1)^{(p-1)/2} \leg{-1}p = 1.
\]
Consequently, each prime dividing every $b_j$ is a quadratic residue\mod{P_{\bf b}}.

Now define, for any integer $N\ge1$,
\[
Q_{N,\bf b} = \prod_{\substack{p\le c\sqrt{\log N} \\ \leg p{P_{\bf b}}=-1}} p,
\]
where $c$ is the constant from Lemma~\ref{lem: montgomery's lemma for G in progressions}.
The prime number theorem for arithmetic progressions to a fixed modulus (see~\cite[Corollary 11.21]{MV}) tells us that 
\[
\log Q_{N,\bf b} \sim \frac12 c\sqrt{\log N}
\]
In particular, $Q_{N,\bf b}$ tends to infinity with $N$, and $Q_{N,\bf b} < e^{c\sqrt{\log N}}$ when $N$ is large enough.

Note also that $Q_{N,\bf b}$ is squarefree and relatively prime to $P_{\bf b}$ and to each $b_j$, since every prime $p \mid b_j$ satisfies $\leg p{P_{\bf b}}=1$. Finally, note that any divisor $d$ of $Q_{N,\bf b}$ has the convenient (and, for us, crucial) property that $\leg d{P_{\bf b}} = \mu(d)$, since both quantities equal $(-1)^{\#\{p\mid d\}}$.
\end{definition}

\begin{lemma} \label{lem:in arithm prog}
For any $1\le b\le P-1$,
we have 
\[
\sum_{\substack{d\mid Q_{N,\bf b} \\ d\equiv b \mod {P_{\bf b}}}} \frac1{d} = \frac{c_{P_{\bf b}}}{\sqrt2(P_{\bf b}-1)} \sqrt{\log\log N} + O_{\bf b}(1),
\]
where $c_P$ was defined in equation~\eqref{cP def}.
\end{lemma}

\begin{proof}
From the orthogonality of the characters modulo~${P_{\bf b}}$,
\begin{align*}
\sum_{\substack{d\mid Q_{N,\bf b} \\ d\equiv b \mod {P_{\bf b}}}} \frac1{d} &= \frac1{{P_{\bf b}}-1} \sum_{\chi \mod {P_{\bf b}}} \bar\chi(b) \sum_{d\mid Q_{N,\bf b}}\frac{\chi(d)}d \\
&= \frac1{{P_{\bf b}}-1} \sum_{\chi \mod {P_{\bf b}}} \bar\chi(b) \prod_{p\mid Q_{N,\bf b}} \bigg( 1 + \frac{\chi(p)}p \bigg) \\
&= \frac1{{P_{\bf b}}-1} \sum_{\chi \mod {P_{\bf b}}} \bar\chi(b) \prod_{\substack{p\le c\sqrt{\log N} \\ \leg p{P_{\bf b}}=-1}} \bigg( 1 + \frac{\chi(p)}p \bigg) \\
&= \frac1{{P_{\bf b}}-1} c_{P_{\bf b}} \sqrt{\log(c\sqrt{\log N})} +O_{P_{\bf b}}(1)
\end{align*}
by Lemma~\ref{half product lemma}. The statement of the lemma follows upon noting that $\log(c\sqrt{\log N}) = \frac12\log\log N+O(1)$.
\end{proof}

\begin{lemma}  \label{K evaluation lemma}
For any integer $m$ that is not a multiple of $P_{\bf b}$ and for any integer $N\ge3$,
\[
K\bigg( Q_{N,\bf b}, \frac{mQ_{N,\bf b}}{P_{\bf b}} \bigg) = \leg m{P_{\bf b}} \leg{Q_{N,\bf b}}{P_{\bf b}} \frac{c_{P_{\bf b}} h(-P_{\bf b})}{\sqrt2(P_{\bf b}-1)} \sqrt{\log\log N}  + O_{\bf b}(1).
\]
\end{lemma}

\begin{remark}
The exact value of the leading constant is not as important for us as the fact that its dependence on $m$ is only in the term $\leg m{P_{\bf b}}$, so that the sign of the leading constant depends on whether $m$ is a quadratic residue or nonresidue modulo~$P_{\bf b}$.
\end{remark}

\begin{proof}
By Definition~\ref{Montgomery quantities def} and the coincidence between the Legendre symbol and the M\"obius function on divisors of $Q_{N,\bf b}$ (as noted at the end of Definition~\ref{P and Q definition}), we have
\begin{align*}
-K\bigg( Q_{N,\bf b}, \frac{mQ_{N,\bf b}}{P_{\bf b}} \bigg) &= \sum_{d\mid Q_{N,\bf b}} \frac{\mu(d)}d B\bigg( \frac{mQ_{N,\bf b}}{dP_{\bf b}} \bigg) \\
&= \sum_{d\mid Q_{N,\bf b}} \leg d{P_{\bf b}} \frac1d B\bigg( \frac{mQ_{N,\bf b}}{dP_{\bf b}} \bigg) \\
&= \sum_{d\mid Q_{N,\bf b}} \leg{mQ_{N,\bf b}}{P_{\bf b}} \leg{mQ_{N,\bf b}/d}{P_{\bf b}} \frac1d B\bigg( \frac{mQ_{N,\bf b}}{dP_{\bf b}} \bigg) \\
&= \leg{mQ_{N,\bf b}}{P_{\bf b}} \sum_{a=1}^{P_{\bf b}-1} \leg a{P_{\bf b}} B\bigg( \frac a{P_{\bf b}} \bigg) \sum_{\substack{d\mid Q_{N,\bf b} \\ \frac{mQ_{N,\bf b}}{d}\equiv a\mod{P_{\bf b}}}} \frac1d.
\end{align*}
The last congruence is equivalent to $d$ being in the reduced residue class 
$\frac{mQ_{N,\bf b}}{a}    \mod{P_{\bf b}}$, and so Lemma~\ref{lem:in arithm prog} applies:
\begin{align*}
-K\bigg( Q_{N,\bf b}, \frac{mQ_{N,{\bf b}}}{P_{\bf b}} \bigg) &= \leg{mQ_{N,\bf b}}{P_{\bf b}} \sum_{a=1}^{P_{\bf b}-1} \leg a{P_{\bf b}} B\bigg( \frac a{P_{\bf b}} \bigg) \bigg( \frac{c_{P_{\bf b}}}{\sqrt2(P_{\bf b}-1)} \sqrt{\log\log N} + O_{\bf b}(1) \bigg) \\
&= \leg{mQ_{N,\bf b}}{P_{\bf b}} \frac{c_{P_{\bf b}}}{\sqrt2(P_{\bf b}-1)} \sqrt{\log\log N} \sum_{a=1}^{P_{\bf b}-1} \leg a{P_{\bf b}} B\bigg( \frac a{P_{\bf b}} \bigg)  + O_{\bf b}(1).
\end{align*}
By the definition of the Bernoulli polynomial $B$, this sum equals
\[
\sum_{a=1}^{P_{\bf b}-1} \leg a{P_{\bf b}} B\bigg( \frac a{P_{\bf b}} \bigg) = \frac1{P_{\bf b}} \sum_{a=1}^{P_{\bf b}-1} a \leg a{P_{\bf b}} - \frac12 \sum_{a=1}^{P_{\bf b}-1} \leg a{P_{\bf b}} = -h(-P_{\bf b}) - 0
\]
by Lemma~\ref{class number lemma}, which completes the proof of the lemma.
\end{proof}

We now have all the tools we need to establish Proposition~\ref{prop: Totient key proposition} and hence our main theorems.

\begin{proof}[Proof of Proposition~\ref{prop: Totient key proposition}]
Given a sufficiently large integer $N$, define two numbers 
\[
\alpha_N^\pm = \frac{m^\pm Q_{N,\bf b}}{P_{\bf b}}, 
\]
where the integers $1\le m^\pm\le P_{\bf b}-1$ satisfy $m^\pm \equiv \pm Q_{N,\bf b} \mod{P_{\bf b}}$; note that none of the numbers $\alpha_N^\pm b_j$ is an integer, since neither $Q_{N,\bf b}$ nor any of the $b_j$ is a multiple of the prime~$P_{\bf b}$. Since we confirmed in Definition~\ref{P and Q definition} that $Q_{N,\bf b} < e^{c\sqrt{\log N}}$, we may invoke Lemma~\ref{lem: montgomery's lemma for G in progressions}:
\begin{align*}
\sum_{n=1}^{N}G_{\bf a,b}(nQ_{N,\bf b}+\alpha_N^+) &= N \sum_{j=1}^{k}a_jC(Q_{N,\bf b}b_j, \alpha_N^+ b_j)+O\big(N\exp(-c\sqrt{\log N})\big) \\
&= N \sum_{j=1}^{k}a_jK(Q_{N,\bf b}b_j,\alpha_N^+ b_j) \frac{6}{\pi^{2}} \prod_{p\mid Q_{N,\bf b}b_j} \bigg(1-\frac1{p^2} \bigg)^{-1} +O(N).
\end{align*}
As noted in Definition~\ref{P and Q definition}, each $b_j$ is relatively prime to $Q_{N,\bf b}$. Thus by Lemma~\ref{K scale lemma} with our choice of $\alpha_N^\pm$,
\begin{multline*}
\sum_{n=1}^{N}G_{\bf a,b}(nQ_{N,\bf b}+\alpha_N^\pm) = O(N) \\
+ \frac{6}{\pi^{2}} N \prod_{p\mid Q_{N,\bf b}} \bigg(1-\frac1{p^2} \bigg)^{-1} \sum_{j=1}^{k} a_j \prod_{p\mid b_j} \bigg(1-\frac1{p^2} \bigg)^{-1} \sum_{d_1d_2=b_j} \frac{\mu(d_1)}{d_1} K\bigg( Q_{N,\bf b}, \frac{m^\pm d_2Q_{N,\bf b}}{P_{\bf b}} \bigg).
\end{multline*}
By Lemma~\ref{K evaluation lemma},
\begin{multline*}
\sum_{n=1}^{N}G_{\bf a,b}(nQ_{N,\bf b}+\alpha_N^\pm) = O(N) + \frac{6}{\pi^{2}} N \prod_{p\mid Q_{N,\bf b}} \bigg(1-\frac1{p^2} \bigg)^{-1} \sum_{j=1}^{k} a_j \prod_{p\mid b_j} \bigg(1-\frac1{p^2} \bigg)^{-1} \\
\times \sum_{d_1d_2=b_j} \frac{\mu(d_1)}{d_1} \bigg( \leg{m^\pm d_2}{P_{\bf b}} \leg{Q_{N,\bf b}}{P_{\bf b}} \frac{c_{P_{\bf b}} h(-P_{\bf b})}{\sqrt2(P_{\bf b}-1)} \sqrt{\log\log N}  + O_{\bf b}(1) \bigg).
\end{multline*}
As noted in Definition~\ref{P and Q definition}, every prime dividing $b_j$ is a quadratic residue modulo $P_{\bf b}$, which implies that $\leg{d_2}{P_{\bf b}}=1$ always. Moreover, $m^+ \equiv Q_{N,\bf b} \mod{P_{\bf b}}$, so the product of Legendre symbols $\leg{m^+}{P_{\bf b}} \leg{Q_{N,\bf b}}{P_{\bf b}}$ equals $1$; on the other hand, since $-1$ is a quadratic nonresidue modulo $P_{\bf b}$, the product of Legendre symbols $\leg{m^-}{P_{\bf b}} \leg{Q_{N,\bf b}}{P_{\bf b}}$ equals $-1$. Consequently,
\begin{multline*}
\sum_{n=1}^{N}G_{\bf a,b}(nQ_{N,\bf b}+\alpha_N^\pm) = \pm \frac{6}{\pi^{2}} \frac{c_{P_{\bf b}} h(-P_{\bf b})}{\sqrt2(P_{\bf b}-1)} N \sqrt{\log\log N} \prod_{p\mid Q_{N,\bf b}} \bigg(1-\frac1{p^2} \bigg)^{-1} \sum_{j=1}^{k} a_j \\
\times \prod_{p\mid b_j} \bigg(1-\frac1{p^2} \bigg)^{-1} \sum_{d_1d_2=b_j} \frac{\mu(d_1)}{d_1} + O(N).
\end{multline*}
The last sum equals $\sum_{d\mid b_j} \frac{\mu(d)}d = \frac{\phi(b_j)}{b_j}$, which when multiplied by the preceding product becomes $\prod_{p\mid b_j} (1+\frac1p)^{-1}$. In addition, by the same argument as in the proof of Lemma~\ref{half product lemma},
\[
\prod_{p\mid Q_{N,\bf b}} \bigg(1-\frac1{p^2} \bigg)^{-1} = \prod_{\leg p{P_{\bf p}}=-1} \bigg(1-\frac1{p^2} \bigg)^{-1} + O\bigg( \frac1{\sqrt{\log N}} \bigg).
\]
We therefore see that we have established the proposition with
\begin{align}
\kappa_{\bf a,b} &= \frac{6}{\pi^{2}} \frac{c_{P_{\bf b}} h(-P_{\bf b})}{\sqrt2(P_{\bf b}-1)} \prod_{\leg p{P_{\bf p}}=-1} \bigg(1-\frac1{p^2} \bigg)^{-1} \sum_{j=1}^{k} a_j \prod_{p\mid b_j} \bigg(1+\frac1p\bigg)^{-1} \notag \\
&= \frac{e^{\gamma/2}3\sqrt2}{\pi^{5/2}} \frac{h(-P_{\bf b})^{1/2}}{P^{1/4}\sqrt{P_{\bf b}-1}} \prod_{\leg p{P_{\bf p}}=-1} \bigg(1-\frac1{p^2} \bigg)^{-1/2} \sum_{j=1}^{k} a_j \prod_{p\mid b_j} \bigg(1+\frac1p\bigg)^{-1}  \label{innermost sum}
\end{align}
by the definition~\eqref{cP def} of $c_P$.
\end{proof}

\begin{remark}
The condition that all of the $a_j$ are positive is used here only to ensure that this last line is non-zero. In the case where the $a_j$ may be arbitrary, and thus where the innermost sum in equation~\eqref{innermost sum} may equal $0$, it is possible that modifying $Q_{N,\boldsymbol{b}}$ and the above argument could prove the totient independence for any coefficients; however, we succeeded in getting this to work only when $k\leq 3$.

It is interesting to note that the methods of \cite{RB}, which handle primitive points in planar convex regions that have nowhere-vanishing Gaussian curvature, are quite different from the present methods and they involve the analysis of zeta functions. 
\end{remark}

\subsection*{Acknowledgement}
The first author was partially supported by ERC Advanced Research Grant no 267165 (DISCONV) and by Hungarian National Science Grant K 111827.  The fourth author was partially supported by ICERM, the Institute for Computational and Experimental Research in Mathematics, Brown University, and would like to thank the warm hospitality of the first author and the Alfr\'ed R\'enyi Institute of Mathematics, Hungarian Academy of Sciences.

\end{document}